\documentclass[11pt,a4paper]{amsart}

\usepackage[a4paper,left=24mm,right=24mm,top=24mm,bottom=26mm]{geometry}

\usepackage[T1]{fontenc}
\usepackage[utf8]{inputenc}
\usepackage[english]{babel}
\usepackage{textcase}
\usepackage{amsmath,amssymb,amsthm,mathtools}
\usepackage{microtype}
\usepackage[hidelinks]{hyperref}
\newtheorem{theorem}{Theorem}[section]
\newtheorem{proposition}[theorem]{Proposition}
\newtheorem{lemma}[theorem]{Lemma}
\newtheorem{corollary}[theorem]{Corollary}
\theoremstyle{remark}
\newtheorem{remark}[theorem]{Remark}

\newcommand{\C}{\mathbb C}
\newcommand{\Z}{\mathbb Z}
\newcommand{\PP}{\mathbb P}
\newcommand{\cO}{\mathcal O}
\newcommand{\ii}{\imath}

\title[HODGE LEVEL AND EFFECTIVE NON-VANISHING]{Hodge Level and Effective Non-vanishing for Smooth Weighted Complete Intersections}
\author{Victor Przyjalkowski}
\address{Steklov Mathematical Institute of Russian Academy of Sciences, 8 Gubkina street, Moscow 119991, Russia}
\email{victorprz@mi-ras.ru, victorprz@gmail.com}
\date{}

\begin{document}

\begin{abstract}
We prove that for every smooth well formed weighted complete intersection of general type of dimension $n>0$, the Hodge number $h^{0,n}(X)$ is positive; in other words, its Hodge level is maximal. We also obtain an explicit lower bound for the geometric genus $p_g(X)$. This implies that the only weighted complete intersection of general type that is not a numerical intersection with a linear cone with $p_g(X)=1$ is $X_{6,6}\subset\PP(1,2,2,3,3)$.
We also prove the Ambro--Kawamata effective non-vanishing conjecture for smooth well formed weighted complete intersections.
\end{abstract}

\maketitle

\section{Introduction}

The Hodge numbers are among the most basic biregular invariants of a smooth projective variety. Following~\cite{Rapoport,Carlson,PSh2020}, we define the \emph{Hodge level} of a smooth projective variety $X$ by
\[
h(X)=\max\{q-p\mid h^{p,q}(X)\ne0\}.
\]

Let $n=\dim X$. By definition, the Hodge level takes its maximal possible value $n$ if and only if the geometric genus $p_g(X)=h^{0,n}(X)$ is nonzero, that is, if and only if $H^0(X,\Omega_X^n)\ne0$.
The Hodge level of any Fano variety is not maximal, since in this case $h^{0,n}(X)=h^0(X,K_X)=0$. On the other hand for a Calabi--Yau variety one has $h^{0,n}(X)=1$, so its Hodge level is maximal. The same clearly holds for curves of positive genus. The Hodge level of a variety of general type need not be maximal even for surfaces: there exist fake projective planes with the same Hodge diamond as $\PP^2$, so their Hodge level is zero~\cite{Mumford,CartwrightSteger}.

The Hodge level of smooth Fano weighted complete intersections was studied in~\cite{PSh2020}. For smooth well formed weighted complete intersections of general type, it was proved in~\cite{Prz2022} that the Hodge level is maximal for complete intersections of Cartier divisors; without the Cartier divisor assumption, the same was proved in codimension~2. The goal of this paper is to prove this for all smooth weighted complete intersections of general type (in other words, to give a positive answer to~\cite[Question 8.6.1]{PSbook}). Related effective non-vanishing results in low codimension were obtained in~\cite{Passantino}.

We prove the following.

\begin{theorem}\label{thm:main}
Let $X$ be a smooth well formed weighted complete intersection of general type of dimension $n>0$. Then
\[
h^{0,n}(X)>0.
\]
In particular, the Hodge level of $X$ is maximal:
\[
h(X)=n.
\]
\end{theorem}

\begin{remark}
The smoothness assumption in Theorem~\ref{thm:main} is essential. It cannot be replaced by quasismoothness: by~\cite[Proposition~4.3]{Prz2022}, for every $n>2$ there exists an $n$-dimensional quasismooth well formed weighted complete intersection of general type such that $h^{0,n}(X)=0$.
\end{remark}

Theorem~\ref{thm:main} yields explicit lower bounds for the geometric genus $p_g(X)$; they are stated in Corollary~\ref{cor:pgbound}. Corollary~\ref{cor:pgone} shows that if $X$ is not a numerical intersection with a linear cone, then $p_g(X)=1$ if and only if $X=X_{6,6}\subset\PP(1,2,2,3,3)$; in particular, in all other cases $p_g(X)\ge2$.

We also prove the Ambro--Kawamata effective non-vanishing conjecture for smooth well formed weighted complete intersections; see Corollary~\ref{cor:AmbroKawamata}.

The idea of the proof of Theorem~\ref{thm:main} is the following. By~\cite[Proposition~3.8.1]{PSbook}, smoothness imposes numerical restrictions on the weights and degrees of $X$. These conditions imply that the rational function
\[
Q(t)=\frac{\prod_j(1+t+\cdots+t^{d_j-1})}
          {\prod_i(1+t+\cdots+t^{a_i-1})}
\]
is a product of cyclotomic polynomials with nonnegative exponents. On the other hand, the Hilbert series of its homogeneous coordinate ring is
\[
\frac{Q(t)}{(1-t)^{n+1}}.
\]
The key point is to establish the positivity of the coefficients of this series starting from a specific one. In particular, the positivity of the first of these coefficients proves the theorem, while the stronger statement is used for effective non-vanishing.

\medskip

Section~2 contains the necessary notation and preliminary results. In Section~3 we construct the cyclotomic numerator of the Hilbert series. Section~4 proves Theorem~\ref{thm:main}, and Section~5 collects its consequences.

\section{Preliminaries}

We recall the necessary notation and results; see~\cite{PSbook} for details.

Let
\[
X\subset \PP(a_0,\ldots,a_N)
\]
be a smooth well formed weighted complete intersection of general type of multidegree $(d_1,\ldots,d_k)$ and dimension
\[
n=N-k>0.
\]
Set
\[
\iota_X=\sum_{j=1}^k d_j-\sum_{i=0}^N a_i.
\]
By the adjunction formula~\cite[Theorem~5.3.2]{PSbook}, we have $\omega_X\simeq \cO_X(\iota_X)$, so $X$ is of general type if and only if $\iota_X>0$.

Set
\[
S=\C[x_0,\ldots,x_N],\qquad \deg x_i=a_i,
\]
and let $f_1,\ldots,f_k$ be a regular sequence defining $X$, with $\deg f_j=d_j$. Set
\[
A=S/(f_1,\ldots,f_k)=\bigoplus_{m\ge0} A_m.
\]

By the adjunction formula and the description of global sections~\cite[Corollary~5.1.13(i)]{PSbook}, we have
\begin{equation}\label{eq:hodgegraded}
h^{0,n}(X)=h^0(\omega_X)=h^0(\cO_X(\iota_X))=\dim_{\C}A_{\iota_X}.
\end{equation}
Therefore, to prove Theorem~\ref{thm:main} it suffices to show that $A_{\iota_X}\ne0$.

Finally, smoothness of $X$ implies the following.

\begin{proposition}[{[PSh26, Proposition~3.8.1]}]\label{prop:divisibility}
Suppose that $r$ weights $a_i$ have a common divisor $\delta>1$. Then there exist $r$ distinct degrees $d_j$ divisible by $\delta$.
\end{proposition}

\section{Hilbert series and the cyclotomic numerator}

Consider the Hilbert series
\[
\sum_{m\ge0}\dim_{\C}A_m\,t^m
\]
of the graded algebra $A$.
Since the polynomials $f_1,\ldots,f_k$ form a regular sequence
\begin{equation}\label{eq:hilb}
\sum_{m\ge0}\dim_{\C}A_m\,t^m=
\frac{\prod_{j=1}^k(1-t^{d_j})}{\prod_{i=0}^N(1-t^{a_i})}
\end{equation}
by~\cite[Lemma~A.2.15]{PSbook}.

For a positive integer $b$, set
\[
U_b(t)=1+t+\cdots+t^{b-1}=\frac{1-t^b}{1-t},
\]
and consider the rational function
\begin{equation}\label{eq:Qdef}
Q(t)=\frac{\prod_{j=1}^k U_{d_j}(t)}{\prod_{i=0}^N U_{a_i}(t)}.
\end{equation}
Let $\Phi_m(t)$ denote the $m$-th cyclotomic polynomial. In particular,
\begin{equation}\label{eq:Ubcycl}
U_b(t)=\prod_{\substack{m\mid b\\m>1}}\Phi_m(t).
\end{equation}

\begin{lemma}\label{lem:Qpoly}
The function $Q(t)$ is a polynomial with integer coefficients and has the form
\[
Q(t)=\prod_{m\ge2}\Phi_m(t)^{c_m},\qquad c_m\in\Z_{\ge0}.
\]
\end{lemma}

\begin{proof}
For $m\ge2$, set
\[
u_m=\#\{i\mid m\mid a_i\},\qquad
v_m=\#\{j\mid m\mid d_j\}.
\]
If $u_m>0$, consider all weights divisible by $m$. Their greatest common divisor is also divisible by $m$. By Proposition~\ref{prop:divisibility}, there exist $u_m$ distinct degrees, each divisible by this greatest common divisor and hence by $m$. Therefore $v_m\ge u_m$.
From~\eqref{eq:Ubcycl} we obtain
\[
Q(t)=\prod_{m\ge2}\Phi_m(t)^{v_m-u_m}\in\Z[t],
\]
which completes the proof.
\end{proof}

It follows from~\eqref{eq:hilb} and~\eqref{eq:Qdef} that
\begin{equation}\label{eq:hilbQ}
\sum_{m\ge0}\dim_{\C}A_m\,t^m=\frac{Q(t)}{(1-t)^{n+1}}.
\end{equation}
The degree of $Q$ is
\begin{equation}\label{eq:q}
q=\sum_{j=1}^k(d_j-1)-\sum_{i=0}^N(a_i-1)
=\iota_X+n+1.
\end{equation}
In particular,
\begin{equation}\label{eq:rhoq}
2\le n+1<q,
\end{equation}
since $n>0$ and $\iota_X>0$.

\section{Proof of Theorem 1.1}

To prove Theorem~\ref{thm:main}, we will use the following lemma.

Let
\[
Q(t)=\prod_{m\ge2}\Phi_m(t)^{c_m},\qquad c_m\in\Z_{\ge0},
\]
and let $q=\deg Q$. For an integer $\rho$ satisfying
\[
2\le\rho<q,
\]
define the coefficients $c_j(Q,\rho)$ by
\begin{equation}\label{eq:cdef}
\frac{Q(t)}{(1-t)^\rho}=\sum_{j\ge0}c_j(Q,\rho)t^j.
\end{equation}

\begin{lemma}\label{lem:coeff}

One has
\[
c_j(Q,\rho)>0\qquad\text{for every }j\ge q-\rho.
\]

\end{lemma}

\begin{proof}
Set
$\ii=q-\rho>0$.
The exponent of the factor $\Phi_2(t)=1+t$ in $Q$ is $c_2$. All the remaining roots of $Q$ split into complex-conjugate pairs $e^{\pm i\theta_\nu}$. Hence
\begin{equation}\label{eq:factor}
Q(t)=(1+t)^{c_2}\prod_{\nu=1}^{r}\bigl((1-t)^2+\lambda_\nu t\bigr),
\qquad
\lambda_\nu=2(1-\cos\theta_\nu)>0,
\end{equation}
where $q=c_2+2r$.

If $r=0$, then $Q(t)=(1+t)^q$. The polynomial $(1+t)^q$ has nonnegative coefficients and nonzero constant term, while the series $(1-t)^{-\rho}$ has strictly positive coefficients. Hence all coefficients of their product are strictly positive, and the assertion follows. Thus we may assume that $r\ge1$.

We have
\[
\frac{Q(t)}{(1+t)^{c_2}}\in\Z[t].
\]
Substituting $t=1$ into~\eqref{eq:factor}, we obtain
\[
\prod_{\nu=1}^{r}\lambda_\nu
=
\left.\frac{Q(t)}{(1+t)^{c_2}}\right|_{t=1}
\in\Z_{>0}.
\]
In particular,
\begin{equation}\label{eq:lambdaprod}
\prod_{\nu=1}^{r}\lambda_\nu\ge1.
\end{equation}

Let $e_v$ be the $v$-th elementary symmetric function in $\lambda_1,\ldots,\lambda_r$, with $e_0=1$. %For $1\le v\le r$, set
By Maclaurin's inequality (see, for example,~\cite[Theorem~52]{HLP})
and by~\eqref{eq:lambdaprod}
one has
\begin{equation}\label{eq:Ev}
e_v\ge \binom rv
\left(\prod_{\nu=1}^{r}\lambda_\nu\right)^{v/r}
\ge\binom rv.
\end{equation}

Since $1+t=(1-t)+2t$,
\begin{equation}\label{eq:expand1}
(1+t)^{c_2}
=\sum_{u=0}^{c_2}\binom{c_2}{u}2^u t^u(1-t)^{c_2-u}.
\end{equation}
Moreover,
\begin{equation}\label{eq:expand2}
\frac{Q(t)}{(1+t)^{c_2}}
=\prod_{\nu=1}^{r}\bigl((1-t)^2+\lambda_\nu t\bigr)
=\sum_{v=0}^{r}e_v t^v(1-t)^{2(r-v)}.
\end{equation}
Multiplying the corresponding sides of~\eqref{eq:expand1} and~\eqref{eq:expand2} and dividing by $(1-t)^\rho$, we obtain
\begin{equation}\label{eq:double}
\frac{Q(t)}{(1-t)^\rho}
=
\sum_{u=0}^{c_2}\sum_{v=0}^{r}
\binom{c_2}{u}2^u e_v\,
 t^{u+v}(1-t)^{\ii-u-2v}.
\end{equation}

Fix $j\ge\ii$. For $v\ge1$, the contribution of each summand on the right-hand side of~\eqref{eq:double} to the coefficient of $t^j$ is nonnegative. Indeed, if $\ii-u-2v\ge0$, then this summand is a polynomial of degree
\[
(u+v)+(\ii-u-2v)=\ii-v<\ii\le j
\]
and therefore contributes nothing. If $\ii-u-2v<0$, then the factor $(1-t)^{\ii-u-2v}$ expands as a formal power series with nonnegative coefficients. Consequently, for every $j\ge\ii$, the coefficient $c_j(Q,\rho)$, viewed as a linear function of $e_1,\ldots,e_r$, is nondecreasing in each $e_v$.

Set
\[
Q_*(t)=(1+t)^{c_2}(1-t+t^2)^r
\]
and, analogously to~(8), define $c_j(Q_*,\rho)$ by
\begin{equation}\label{eq:Qstarseries}
\frac{Q_*(t)}{(1-t)^\rho}
=\sum_{j\ge0}c_j(Q_*,\rho)t^j.
\end{equation}

Using~\eqref{eq:Ev}, we may replace $e_v$ by $\binom rv$ in the right-hand side of~\eqref{eq:double}. We obtain the series~\eqref{eq:Qstarseries}; hence
\[
c_j(Q,\rho)\ge c_j(Q_*,\rho)
\qquad\text{for every }j\ge\ii.
\]
Thus it suffices to prove that
\begin{equation}\label{eq:targetstar}
c_j(Q_*,\rho)>0
\qquad\text{for every }j\ge\ii.
\end{equation}

Set
\[
\Psi(t)=1-t+t^2,
\]
so that
\begin{equation}\label{eq:Psi}
(1+t)\Psi(t)=1+t^3.
\end{equation}

Assume that
$c_2\ge r$.  %}
Then from~\eqref{eq:Psi} it follows that
\[
Q_*(t)=(1+t)^{c_2-r}(1+t^3)^r.
\]

This polynomial has nonnegative coefficients and constant term $1$. Since the series $(1-t)^{-\rho}$ has strictly positive coefficients, every coefficient of $Q_*(t)/(1-t)^\rho$ is strictly positive. Hence~\eqref{eq:targetstar} holds in this case.

Thus we may assume that $c_2<r$. %}
Set
\[
M=r-c_2>0.
\]
By~\eqref{eq:Psi}, we have
\begin{equation}\label{eq:QstarPsi}
Q_*(t)=(1+t^3)^{c_2}\Psi(t)^M.
\end{equation}

First suppose that $\rho\ge M$. Then
\begin{equation}\label{eq:rhoM}
\frac{Q_*(t)}{(1-t)^\rho}
=(1+t^3)^{c_2}
\left(\frac{\Psi(t)}{1-t}\right)^M
(1-t)^{-(\rho-M)}.
\end{equation}
The series
\[
\frac{\Psi(t)}{1-t}
=1+\frac{t^2}{1-t}
=1+t^2+t^3+t^4+\cdots
\]
has nonnegative coefficients. If $\rho>M$, then the last factor on the right-hand side of~\eqref{eq:rhoM} has strictly positive coefficients in every degree, and~\eqref{eq:targetstar} follows immediately.

Let $\rho=M$. Then
\[
\ii=q-\rho=c_2+2r-M=3c_2+M.
\]
For every $j\ge\ii$, choose the term $t^{3c_2}$ from $(1+t^3)^{c_2}$. The remaining degree satisfies
\[
j-3c_2\ge M\ge2.
\]
Since $1+t^2+t^3+\cdots$ has a positive coefficient in every degree at least $2$, the coefficient of $t^{j-3c_2}$ in $(1+t^2+t^3+\cdots)^M$ is positive. Hence~\eqref{eq:targetstar} also holds when $\rho=M$.

It remains to consider the case
\begin{equation}\label{eq:hardcase}
2\le\rho<M.
\end{equation}
Set
\[
B(t)=\frac{\Psi(t)^M}{(1-t)^\rho}
=\sum_{j\ge0}b_jt^j,
\qquad
a=2M-\rho.
\]
By~\eqref{eq:QstarPsi}, we have
\begin{equation}\label{eq:QB}
\frac{Q_*(t)}{(1-t)^\rho}=(1+t^3)^{c_2}B(t).
\end{equation}
We claim that
\begin{equation}\label{eq:bj}
b_j>0\qquad\text{for all }j\ge a.
\end{equation}

Since $\Psi(t)=(1-t)^2+t$, we have
\begin{equation}\label{eq:Bexpand}
B(t)=\sum_{v=0}^{M}\binom Mv t^v(1-t)^{a-2v}.
\end{equation}

First let $j>a$. Consider the coefficient $b_j$. The summands on the right-hand side of~\eqref{eq:Bexpand} with $a-2v\ge0$ are polynomials of degree $a-v\le a$ and do not contribute to $b_j$. The summands with $a-2v<0$ expand as series with nonnegative coefficients. Finally, the summand corresponding to $v=M$ is
\[
t^M(1-t)^{-\rho}
\]
and gives a strictly positive contribution to $b_j$, since $j>a=2M-\rho>M$ implies $j-M>0$. Hence $b_j>0$ for $j>a$.

Now let $j=a$. The summand with $v=0$ in~\eqref{eq:Bexpand} contributes $(-1)^a$. All summands with $v\ge1$ and $a-2v\ge0$ have degree $a-v<a$, while those with $a-2v<0$ give nonnegative contributions. If $a$ is even, the contribution at $v=0$ is already $1$, and $b_a>0$.
Let $a$ be odd. Then $\rho$ is also odd, and $\rho\ge2$ implies $\rho\ge3$. Set
\[
v_0=\frac{a+1}{2}=M-\frac{\rho-1}{2}.
\]
In particular,~\eqref{eq:hardcase} implies $1\le v_0\le M-1$. Since $a-2v_0=-1$, the corresponding summand on the right-hand side of~\eqref{eq:Bexpand} is
\[
\binom M{v_0}t^{v_0}(1-t)^{-1},
\]
and its contribution to the coefficient of $t^a$ equals $\binom M{v_0}>1$. This strictly exceeds the only possible negative contribution, namely $-1$ from $v=0$. Thus~\eqref{eq:bj} is proved.

Finally, $q=c_2+2r=3c_2+2M$ implies
\[
\ii=q-\rho=3c_2+a.
\]
By~\eqref{eq:QB}, for every $j\ge\ii$ we have
\[
c_j(Q_*,\rho)
=
\sum_{s=0}^{c_2}\binom{c_2}{s}b_{j-3s}.
\]
Since $0\le s\le c_2$,
\[
j-3s\ge j-3c_2\ge\ii-3c_2=a,
\]
and therefore all the coefficients $b_{j-3s}$ occurring in the sum are positive by~\eqref{eq:bj}. Hence $c_j(Q_*,\rho)>0$ for every $j\ge\ii$. This proves~\eqref{eq:targetstar}, and together with the comparison above completes the proof.

\end{proof}

\begin{remark}\label{rem:rho1}
The condition $\rho\ge2$ is essential. For $\rho=1$ the statement is false: for
\[
Q(t)=\Phi_6(t)=1-t+t^2
\]
we have
\[
\frac{Q(t)}{1-t}=1+0\cdot t+t^2+t^3+t^4+\cdots,
\]
so the coefficient of $t$ is zero. However, in Theorem~\ref{thm:main} we apply Lemma~\ref{lem:coeff} with $\rho=n+1\ge2$.
\end{remark}

We now prove the main theorem.

\begin{proof}[Proof of Theorem~\ref{thm:main}]
Set $\rho=n+1$. By Lemma~\ref{lem:Qpoly} and formulas~\eqref{eq:hilbQ} and~\eqref{eq:q}, we have
\[
\dim_{\C}A_{\iota_X}=c_{q-\rho}(Q,\rho).
\]
By~\eqref{eq:rhoq} and Lemma~\ref{lem:coeff}, the right-hand side is positive. Therefore, by~\eqref{eq:hodgegraded},
\[
h^{0,n}(X)=\dim_{\C}A_{\iota_X}>0,
\]
which completes the proof.
\end{proof}

\begin{remark}[{cf.~\cite[Corollary 3.9]{Prz2022}}]
From $A_{\iota_X}\ne0$ it follows that there exist positive integers $\beta_0,\ldots,\beta_N$ such that
\[
\sum_{j=1}^k d_j=\sum_{i=0}^N\beta_i a_i.
\]
Indeed, $A_{\iota_X}\ne0$ implies $S_{\iota_X}\ne0$, so there is a monomial $x_0^{\alpha_0}\cdots x_N^{\alpha_N}$ of degree $\iota_X$. Setting $\beta_i=\alpha_i+1$, we obtain the required equality.
\end{remark}

\section{Consequences of Theorem 1.1}

In this section we derive several consequences of Theorem~\ref{thm:main} and its proof.
Let $X\subset\PP(a_0,\ldots,a_N)$ be a smooth well formed weighted complete intersection of general type of multidegree $(d_1,\ldots,d_k)$ and dimension $n$.

Set
\[
c_2=
\left|\{j: 2\mid d_j\}\right|
-
\left|\{i: 2\mid a_i\}\right|.
\]
Recall that Proposition~\ref{prop:divisibility} gives $c_2\ge0$. In the notation of~\eqref{eq:factor}, we have $q=c_2+2r$.

\begin{corollary}\label{cor:pgbound}
One has
\[
p_g(X)\ge
\begin{cases}
\displaystyle\binom{\iota_X+n}{n}, & r\le c_2,\\[7pt]
\displaystyle\binom{\iota_X+n-r+c_2}{n-r+c_2}, & 1\le r-c_2\le n,\\[7pt]
n+1, & r=c_2+n+1,\\[4pt]
2^{\max\{c_2,1\}}, & r\ge c_2+n+2.
\end{cases}
\]
\end{corollary}

\begin{proof}
Set $M=r-c_2$. By~\eqref{eq:hodgegraded}, \eqref{eq:hilbQ}, \eqref{eq:rhoq}, and the comparison obtained after~\eqref{eq:Qstarseries} in the proof of Lemma~\ref{lem:coeff}, we have
\[
p_g(X)=h^{0,n}(X)=c_{\iota_X}(Q,n+1)
\ge c_{\iota_X}(Q_*,n+1),
\]
where
\[
Q_*(t)=(1+t)^{c_2}(1-t+t^2)^r.
\]

First let %$M\le0$. Then 
$c_2\ge r$. Then~\eqref{eq:Psi} gives
\[
Q_*(t)=(1+t)^{c_2-r}(1+t^3)^r.
\]
All coefficients of this polynomial are nonnegative, and its constant term is $1$. Hence the coefficient of $t^{\iota_X}$ in the series $Q_*(t)/(1-t)^{n+1}$ is at least the coefficient of the same degree in $(1-t)^{-n-1}$. Therefore
\[
p_g(X)\ge\binom{\iota_X+n}{n}.
\]

Let $1\le M\le n$. From~\eqref{eq:rhoM} we obtain
\[
\frac{Q_*(t)}{(1-t)^{n+1}}
=
(1+t^3)^{c_2}
(1+t^2+t^3+\cdots)^M
(1-t)^{-(n+1-M)}.
\]
The first two factors have nonnegative coefficients and constant term $1$. Therefore
\[
p_g(X)\ge\binom{\iota_X+n-M}{n-M}=\binom{\iota_X+n-r+c_2}{n-r+c_2}.
\]

Let $M=n+1$. This is the case $\rho=M$ in the proof of Lemma~\ref{lem:coeff}; there we obtained
\[
\iota_X=3c_2+M.
\]
The coefficient of $t^M$ in $(1+t^2+t^3+\cdots)^M$ is at least $M$: in each of the $M$ factors one can choose the term $t^M$, and choose the constant terms in all the others. Together with the term $t^{3c_2}$ from $(1+t^3)^{c_2}$, this gives
\[
p_g(X)\ge M=n+1.
\]

Finally, let $M\ge n+2$. Then~\eqref{eq:hardcase} holds. In the notation of~\eqref{eq:QB}--\eqref{eq:Bexpand},
\[
a=2M-n-1,
\]
and at the end of the proof of Lemma~\ref{lem:coeff} we obtained
\[
c_{\iota_X}(Q_*,n+1)
=
\sum_{s=0}^{c_2}\binom{c_2}{s}
 b_{a+3(c_2-s)},
\]
and by the inequalities in~\eqref{eq:bj}, all $b_j$ with $j\ge a$ are positive integers. If $c_2>0$, then
\[
p_g(X)\ge
\sum_{s=0}^{c_2}\binom{c_2}{s}
=2^{c_2}.
\]

Let $c_2=0$. Then $p_g(X)\ge b_a$. If $a$ is odd, inequality~\eqref{eq:bj} in Lemma~\ref{lem:coeff} gives
\[
b_a\ge
\binom{M}{(a+1)/2}-1
\ge M-1\ge2.
\]
If $a$ is even, the summand with $v=0$ in~\eqref{eq:Bexpand} contributes $1$, while the summand with $v=a/2+1$ gives one more positive contribution. Hence $b_a\ge2$ in this case as well. Therefore
{\renewcommand{\qedsymbol}{\openbox}
\begin{equation*}
p_g(X)\ge2^{\max\{c_2,1\}}.\qedhere
\end{equation*}
}
\end{proof}

\begin{corollary}\label{cor:pgone}
Assume that $X$ is not a numerical intersection with a linear cone, that is, $d_j\ne a_i$ for all $i,j$. Then
\[
p_g(X)=1
\]
if and only if
\[
X=X_{6,6}\subset\PP(1,2,2,3,3).
\]
In particular, in all other cases
\[
p_g(X)\ge2.
\]
\end{corollary}

\begin{proof}
Assume that $p_g(X)=1$. By Corollary~\ref{cor:pgbound}, we have $r-c_2=n$, so~\eqref{eq:rhoM} implies
\[
\frac{Q_*(t)}{(1-t)^{n+1}}
=
\frac{(1+t^3)^{c_2}(1+t^2+t^3+\cdots)^n}{1-t}.
\]
Moreover,~\eqref{eq:q}, \eqref{eq:rhoq}, \eqref{eq:factor}, and $r=c_2+n$ give
\[
\iota_X=3c_2+n-1.
\]
From the inequality $c_{\iota_X}(Q,n+1)\ge c_{\iota_X}(Q_*,n+1)$ and inequality~\eqref{eq:targetstar}, we obtain
\[
1=p_g(X)\ge c_{\iota_X}(Q_*,n+1)>0,
\]
hence
\[
c_{\iota_X}(Q_*,n+1)=1.
\]

Set
\[
F(t)=(1+t^3)^{c_2}(1+t^2+t^3+\cdots)^n.
\]
All coefficients of the series $F(t)$ are nonnegative, its constant term is $1$, and the coefficient of $t^2$ is $n>0$. The coefficient of $t^{\iota_X}$ in $F(t)/(1-t)$ is the sum of the coefficients of the series $F(t)$ in degrees from $0$ to $\iota_X$. Hence, if $\iota_X\ge2$, it would be greater than $1$. Thus
\[
\iota_X=1.
\]
Hence,
\[
3c_2+n=2.
\]
Since $c_2\ge0$ and $n>0$, we obtain
\[
c_2=0,
\qquad
n=2.
\]
Thus,
\[
r=M=2,
\qquad
q=4.
\]

Formula~\eqref{eq:factor} becomes
\[
Q(t)=
\bigl((1-t)^2+\lambda_1t\bigr)
\bigl((1-t)^2+\lambda_2t\bigr).
\]
Considering the coefficient of $t$ on the right-hand side of~\eqref{eq:double}, we obtain
\[
p_g(X)=e_1-1.
\]
Since $p_g(X)=1$,
\[
e_1=\lambda_1+\lambda_2=2.
\]
On the other hand, by~\eqref{eq:lambdaprod},
\[
e_2=\lambda_1\lambda_2\ge 1. %=Q(1)\in\Z_{>0}.
\]
Since
\[
2=e_1=\lambda_1+\lambda_2
\ge2\sqrt{\lambda_1\lambda_2}=2\sqrt{e_2}\ge 2,
\]
one has %$e_2\le1$. Therefore
\[
e_2=1,
\qquad
\lambda_1=\lambda_2=1.
\]
Therefore
\[
Q(t)=(1-t+t^2)^2=\Phi_6(t)^2.
\]
By~\eqref{eq:hilbQ},
\[
\sum_{m\ge0}\dim_{\C}A_m\,t^m
=
\frac{\Phi_6(t)^2}{(1-t)^3}
=
\frac{(1-t^6)^2}
{(1-t)(1-t^2)^2(1-t^3)^2}.
\]

Since $X$ is not a numerical intersection with a linear cone, $d_j\ne a_i$ for all $i,j$. Thus $\mbox{by~\cite[Lemma~5.5.7]{PSbook}}$, we have
\[
(d_1,d_2)=(6,6),
\qquad
(a_0,\ldots,a_4)=(1,2,2,3,3).
\]
Hence,
\[
X=X_{6,6}\subset\PP(1,2,2,3,3).
\]

Conversely, if $X=X_{6,6}\subset\PP(1,2,2,3,3)$ is smooth, then the Hilbert series above gives $p_g(X)=1$. Smooth members of this family exist by~\cite[Corollary~3.8.16]{PSbook}.
\end{proof}

We conclude with an application in which the general type assumption is no longer imposed.

\begin{corollary}[{Ambro--Kawamata effective non-vanishing; cf.~\cite[Conjecture~2.1]{Kawamata}}]
\label{cor:AmbroKawamata}
Let $X$ be a smooth well formed weighted complete intersection, let $H$ be a nef Cartier divisor on $X$, and let $\Delta$ be an effective $\mathbb R$-divisor such that $(X,\Delta)$ is klt. If
\[
H-K_X-\Delta
\]
is nef and big, then
\[
H^0(X,\cO_X(H))\ne0.
\]
In particular, the Ambro--Kawamata effective non-vanishing conjecture holds for smooth well formed weighted complete intersections.
\end{corollary}

\begin{proof}
The cases $n=0$ and $n=1$ are obvious, while the case $n=2$ is proved in~\cite[Theorem~3.1]{Kawamata}. Thus we may assume that $n\ge3$. The Fano and Calabi--Yau cases follow from~\cite[Corollary~5.11]{PST17}. Thus we may assume that $X$ is of general type. Let $L$ be an ample Cartier divisor on $X$ such that
$\cO_X(L)\simeq\cO_X(1)$. Since $X$ is smooth, by~\cite[Proposition~2.3]{PST17} we have
$\operatorname{Pic}(X)=\mathbb Z[L]$. By adjunction, $K_X\sim\iota_XL$.
Write $H\sim hL$ for some $h\in\mathbb Z_{\ge0}$. We have
$\Delta\equiv bL$ for some $b\ge0$. Since
\[
H-K_X-\Delta\equiv(h-\iota_X-b)L
\]
is nef and big, we have
\[
h-\iota_X-b>0,
\]
and in particular $h>\iota_X$.

The Hilbert series formula~\eqref{eq:hilbQ}, with
\[
\rho=n+1,\qquad q=\iota_X+n+1,
\]
and Lemma~\ref{lem:coeff} give
\[
h^0(X,\cO_X(H))=\dim_{\mathbb C}A_h>0,
\]
which implies the assertion of the corollary.\qedhere
\end{proof}

\begin{remark}
Effective non-vanishing for weighted complete intersections was previously established by Pizzato, Sano, and Tasin for quasismooth Fano and Calabi--Yau weighted complete intersections and for quasismooth weighted hypersurfaces~\cite{PST17}, by Jiang and Yu in codimension~2~\cite{JiangYu}, and by Passantino in codimension at most~3~\cite{Passantino}.
\end{remark}

\medskip
\noindent\textbf{Acknowledgements.} The author is grateful to Constantin Shramov for useful discussions.

\medskip
\noindent\textbf{Use of AI.} During the preparation of this paper, the author used AI as an auxiliary tool. The final mathematical verification of the proofs and the writing of the paper were carried out independently by the author.


\begin{thebibliography}{PSh26}

\bibitem[Ca80]{Carlson}
J.\,A.\,Carlson,
\emph{Extensions of mixed Hodge structures},
Journ\'ees de G\'eom\'etrie Alg\'ebrique d'Angers, 1979,
107--127, Sijthoff \& Noordhoff, 1980.

\bibitem[CS10]{CartwrightSteger}
D.\,Cartwright, T.\,Steger,
\emph{Enumeration of the 50 fake projective planes},
C. R. Math. Acad. Sci. Paris 348 (2010), 11--13.

\bibitem[HLP52]{HLP}
G.\,H.\,Hardy, J.\,E.\,Littlewood, G.\,P\'olya,
\emph{Inequalities}, 2nd ed., Cambridge University Press, 1952, Theorem~52.


\bibitem[JY24]{JiangYu}
C.\,Jiang, P.\,Yu,
\emph{Effective nonvanishing for weighted complete intersections of codimension two},
to appear in Kyoto J. Math.; arXiv:2409.07828.

\bibitem[Ka00]{Kawamata}
Y.\,Kawamata,
\emph{On effective non-vanishing and base-point-freeness},
Asian J. Math. 4 (2000), no.~1, 173--182.


\bibitem[Mu79]{Mumford}
D.\,Mumford,
\emph{An algebraic surface with $K$ ample, $K^2=9$, $p_g=q=0$},
Amer. J. Math. 101 (1979), 233--244.

\bibitem[Pa25]{Passantino}
A.\,Passantino,
\emph{Effective non-vanishing for weighted complete intersections of low codimension},
arXiv:2501.13267v1 (2025).


\bibitem[PST17]{PST17}
M.\,Pizzato, T.\,Sano, L.\,Tasin,
\emph{Effective non-vanishing for Fano weighted complete intersections},
Algebra Number Theory 11 (2017), no.~10, 2369--2395.


\bibitem[Pr22]{Prz2022}
V.\,Przyjalkowski,
\emph{On Hodge level of weighted complete intersections of general type},
Sbornik: Mathematics 213:12 (2022), 1679--1694.

\bibitem[PSh20]{PSh2020}
V.\,Przyjalkowski, C.\,Shramov,
\emph{Hodge level for weighted complete intersections},
Collect. Math. 71 (2020), 549--574.

\bibitem[PSh26]{PSbook}
V.\,V.\,Przyjalkowski, C.\,C.\,Shramov,
\emph{Weighted Complete Intersections},
De Gruyter Expositions in Mathematics, vol.~79,
Walter de Gruyter, Berlin--Boston, 2026.

\bibitem[Ra72]{Rapoport}
M.\,Rapoport,
\emph{Compl\'ement \`a l'article de P.\,Deligne ``La conjecture de Weil pour les surfaces K3''},
Invent. Math. 15 (1972), 227--236.

\end{thebibliography}
\end{document}